\documentclass[11pt,reqno, notitlepage]{amsart}
\usepackage{amsmath,amsfonts,amsthm,amssymb}
\usepackage{graphicx}
\usepackage{verbatim}
\usepackage{color}
\usepackage{amscd}
\usepackage{verbatim}

\usepackage[french,english]{babel}

\usepackage[utf8]{inputenc}

\begin{document}
\newcommand{\M}{{\mathcal M}}
\newcommand{\loc}{{\mathrm{loc}}}
\newcommand{\core}{C_0^{\infty}(\Omega)}
\newcommand{\sob}{W^{1,p}(\Omega)}
\newcommand{\sobloc}{W^{1,p}_{\mathrm{loc}}(\Omega)}
\newcommand{\merhav}{{\mathcal D}^{1,p}}
\newcommand{\be}{\begin{equation}}
\newcommand{\ee}{\end{equation}}
\newcommand{\mysection}[1]{\section{#1}\setcounter{equation}{0}}
\newcommand{\laplace}{\Delta}
\newcommand{\pl}{\laplace_p}
\newcommand{\grad}{\nabla}
\newcommand{\pd}{\partial}
\newcommand{\bo}{\pd}
\newcommand{\csub}{\subset \subset}
\newcommand{\sm}{\setminus}
\newcommand{\ssm}{:}
\newcommand{\diver}{\mathrm{div}\,}
\newcommand{\bea}{\begin{eqnarray}}
\newcommand{\eea}{\end{eqnarray}}
\newcommand{\bean}{\begin{eqnarray*}}
\newcommand{\eean}{\end{eqnarray*}}
\newcommand{\thkl}{\rule[-.5mm]{.3mm}{3mm}}
\newcommand{\cw}{\stackrel{\rightharpoonup}{\rightharpoonup}}
\newcommand{\id}{\operatorname{id}}
\newcommand{\supp}{\operatorname{supp}}
\newcommand{\wlim}{\mbox{ w-lim }}
\newcommand{\mymu}{{x_N^{-p_*}}}
\newcommand{\R}{{\mathbb R}}
\newcommand{\N}{{\mathbb N}}
\newcommand{\Z}{{\mathbb Z}}
\newcommand{\Q}{{\mathbb Q}}
\newcommand{\abs}[1]{\lvert#1\rvert}
\newtheorem{theorem}{Theorem}[section]
\newtheorem{corollary}[theorem]{Corollary}
\newtheorem{lemma}[theorem]{Lemma}
\newtheorem{notation}[theorem]{Notation}
\newtheorem{definition}[theorem]{Definition}
\newtheorem{remark}[theorem]{Remark}
\newtheorem{proposition}[theorem]{Proposition}
\newtheorem{assertion}[theorem]{Assertion}
\newtheorem{problem}[theorem]{Problem}
\newtheorem{conjecture}[theorem]{Conjecture}
\newtheorem{question}[theorem]{Question}
\newtheorem{example}[theorem]{Example}
\newtheorem{Thm}[theorem]{Theorem}
\newtheorem{Lem}[theorem]{Lemma}
\newtheorem{Pro}[theorem]{Proposition}
\newtheorem{Def}[theorem]{Definition}
\newtheorem{defi}[theorem]{Definition}
\newtheorem{Exa}[theorem]{Example}
\newtheorem{Exs}[theorem]{Examples}
\newtheorem{Rems}[theorem]{Remarks}
\newtheorem{Rem}[theorem]{Remark}

\newtheorem{Cor}[theorem]{Corollary}
\newtheorem{Conj}[theorem]{Conjecture}
\newtheorem{Prob}[theorem]{Problem}
\newtheorem{Ques}[theorem]{Question}
\newtheorem*{corollary*}{Corollary}
\newtheorem*{theorem*}{Theorem}
\newtheorem{thm}[theorem]{Theorem}
\newtheorem{lem}[theorem]{Lemma}
\newtheorem{prop}[theorem]{Proposition}
\newtheorem{cor}[theorem]{Corollary}
\newtheorem{ex}[theorem]{Example}
\newtheorem{rem}[theorem]{Remark}
\newtheorem{rems}[theorem]{Remarks}
\newtheorem*{thmm}{Theorem}
\newcommand{\Hmm}[1]{\leavevmode{\marginpar{\tiny%
$\hbox to 0mm{\hspace*{-0.5mm}$\leftarrow$\hss}%
\vcenter{\vrule depth 0.1mm height 0.1mm width \the\marginparwidth}%
\hbox to
0mm{\hss$\rightarrow$\hspace*{-0.5mm}}$\\\relax\raggedright #1}}}
\newcommand{\pf}{\noindent \mbox{{\bf Proof}: }}


\renewcommand{\theequation}{\thesection.\arabic{equation}}
\catcode`@=11 \@addtoreset{equation}{section} \catcode`@=12
\newcommand{\Real}{\mathbb{R}}
\newcommand{\real}{\mathbb{R}}
\newcommand{\Nat}{\mathbb{N}}
\newcommand{\ZZ}{\mathbb{Z}}
\newcommand{\CC}{\mathbb{C}}
\newcommand{\Pess}{\opname{Pess}}
\newcommand{\Proof}{\mbox{\noindent {\bf Proof} \hspace{2mm}}}
\newcommand{\mbinom}[2]{\left (\!\!{\renewcommand{\arraystretch}{0.5}
\mbox{$\begin{array}[c]{c}  #1\\ #2  \end{array}$}}\!\! \right )}
\newcommand{\brang}[1]{\langle #1 \rangle}
\newcommand{\vstrut}[1]{\rule{0mm}{#1mm}}
\newcommand{\rec}[1]{\frac{1}{#1}}
\newcommand{\set}[1]{\{#1\}}
\newcommand{\dist}[2]{$\mbox{\rm dist}\,(#1,#2)$}
\newcommand{\opname}[1]{\mbox{\rm #1}\,}
\newcommand{\mb}[1]{\;\mbox{ #1 }\;}
\newcommand{\undersym}[2]
 {{\renewcommand{\arraystretch}{0.5}  \mbox{$\begin{array}[t]{c}
 #1\\ #2  \end{array}$}}}
\newlength{\wex}  \newlength{\hex}
\newcommand{\understack}[3]{%
 \settowidth{\wex}{\mbox{$#3$}} \settoheight{\hex}{\mbox{$#1$}}
 \hspace{\wex}  \raisebox{-1.2\hex}{\makebox[-\wex][c]{$#2$}}
 \makebox[\wex][c]{$#1$}   }%
\newcommand{\smit}[1]{\mbox{\small \it #1}}
\newcommand{\lgit}[1]{\mbox{\large \it #1}}
\newcommand{\scts}[1]{\scriptstyle #1}
\newcommand{\scss}[1]{\scriptscriptstyle #1}
\newcommand{\txts}[1]{\textstyle #1}
\newcommand{\dsps}[1]{\displaystyle #1}
\newcommand{\dx}{\,\mathrm{d}x}
\newcommand{\dy}{\,\mathrm{d}y}
\newcommand{\dz}{\,\mathrm{d}z}
\newcommand{\dm}{\,\mathrm{d}m}
\newcommand{\dt}{\,\mathrm{d}t}
\newcommand{\dr}{\,\mathrm{d}r}
\newcommand{\du}{\,\mathrm{d}u}
\newcommand{\dv}{\,\mathrm{d}v}
\newcommand{\dV}{\,\mathrm{d}V}
\newcommand{\ds}{\,\mathrm{d}s}
\newcommand{\da}{\,\mathrm{d}\alpha}
\newcommand{\db}{\,\mathrm{d}\beta}
\newcommand{\dS}{\,\mathrm{d}S}
\newcommand{\dk}{\,\mathrm{d}k}

\newcommand{\dphi}{\,\mathrm{d}\phi}
\newcommand{\dtau}{\,\mathrm{d}\tau}
\newcommand{\dxi}{\,\mathrm{d}\xi}
\newcommand{\deta}{\,\mathrm{d}\eta}
\newcommand{\dsigma}{\,\mathrm{d}\sigma}
\newcommand{\dtheta}{\,\mathrm{d}\theta}
\newcommand{\dnu}{\,\mathrm{d}\nu}

\def\ga{\alpha}     \def\gb{\beta}       \def\gg{\gamma}
\def\gc{\chi}       \def\gd{\delta}      \def\ge{\epsilon}
\def\gth{\theta}                         \def\vge{\varepsilon}
\def\gf{\phi}       \def\vgf{\varphi}    \def\gh{\eta}
\def\gi{\iota}      \def\gk{\kappa}      \def\gl{\lambda}
\def\gm{\mu}        \def\gn{\nu}         \def\gp{\pi}
\def\vgp{\varpi}    \def\gr{\rho}        \def\vgr{\varrho}
\def\gs{\sigma}     \def\vgs{\varsigma}  \def\gt{\tau}
\def\gu{\upsilon}   \def\gv{\vartheta}   \def\gw{\omega}
\def\gx{\xi}        \def\gy{\psi}        \def\gz{\zeta}
\def\Gg{\Gamma}     \def\Gd{\Delta}      \def\Gf{\Phi}
\def\Gth{\Theta}
\def\Gl{\Lambda}    \def\Gs{\Sigma}      \def\Gp{\Pi}
\def\Gw{\Omega}     \def\Gx{\Xi}         \def\Gy{\Psi}

\renewcommand{\div}{\mathrm{div}}
\newcommand{\red}[1]{{\color{red} #1}}

%


\newcommand{\De} {\Delta}
\newcommand{\la} {\lambda}
\newcommand{\bn}{\mathbb{B}^{2}}
\newcommand{\rn}{\mathbb{R}^{2}}
\newcommand{\bnn}{\mathbb{B}^{N}}
\newcommand{\rnn}{\mathbb{R}^{N}}
\newcommand{\kP}{k_{P}^{M}}
\newcommand{\gpm}{G_{P}^{M}}
\newcommand{\Lp}{L_{\Phi}}
\newcommand{\gmj}{G^{M_{j}}_{P}(x, p)}
\newcommand{\ojr}{\gw_{j}(R_{1}, R_{2})}
\newcommand{\oj}{\gw_{j}}

\newcommand{\authorfootnotes}{\renewcommand\thefootnote{\@fnsymbol\c@footnote}}%

\def\e{{\text{e}}}
\def\N{{I\!\!N}}

\numberwithin{equation}{section} \allowdisplaybreaks

\title[Green functions of critical operators on Manifolds]{On Green functions of second-order elliptic operators on Riemannian Manifolds:\\ the critical case}

\author{Debdip Ganguly}
\address{Debdip Ganguly, Department of Mathematics,  Technion - Israel Institute of Technology, Haifa 32000, Israel}
\email{gdebdip@technion.ac.il}
\author{Yehuda Pinchover}
\address{Yehuda Pinchover,
Department of Mathematics, Technion - Israel Institute of
Technology,   Haifa 32000, Israel}
\email{pincho@technion.ac.il}

\date{}

\begin{abstract}
Let $P$ be a second-order, linear, elliptic operator with real coefficients which is defined on a noncompact and connected Riemannian manifold $M$.  It is well known that the equation $Pu=0$ in $M$ admits a positive supersolution which is not a solution if and only if $P$ admits a unique positive minimal Green function on $M$, and in this case, $P$ is said to be {\em subcritical} in $M$. If $P$ does not admit a positive Green function but admits a global positive (super)solution, then such a solution is called a {\em ground state} of $P$ in $M$, and $P$ is said to be {\em critical} in $M$.

We prove for a critical operator $P$ in $M$, the existence of a Green function which is dominated above by the ground state of $P$ away from the singularity. Moreover, in a certain class of Green functions, such a Green function is unique, up to an addition of a product of the ground states of $P$ and $P^\star$.
Under some further assumptions, we describe the behaviour at infinity of such a Green function. This result extends and sharpens the celebrated result of P.~Li and L.-F.~Tam concerning the existence of a {\em symmetric} Green function for the Laplace-Beltrami operator on a smooth and  {\em complete} Riemannian manifold $M$.

\vspace{.2cm}

\noindent  2000  \! {\em Mathematics  Subject  Classification.}
{Primary 35J08 ; Secondary 31C35, 35A08, 35B09, 58G03.} \\[1mm]
\end{abstract}
\maketitle
 \section{Introduction}\label{sec_int}\label{sec_1}
Let $M$ be a noncompact and connected  manifold of dimension $N\geq 2$ and of class $C^2$. We assume that $\nu$ is a positive measure on $M$, satisfying
$\dnu=f\, \mbox{d}\mathrm{vol}$, where $f$ is a strictly positive function and $\mathrm{vol}$ is the volume form of $M$.
On $M$ we consider a second-order elliptic operator $P$ with real coefficients which (in any coordinate system $(U;x_{1},\ldots,x_{N})$) is of the divergence form
\begin{equation} \label{div_P}
Pu:=-\div \left[\big(A(x)\nabla u +  u\tilde{b}(x) \big) \right]  +
 b(x)\cdot\nabla u   +c(x)u.
\end{equation}
Here, the minus divergence is the formal adjoint of the gradient with respect to the measure $\nu$.  We assume that for every $x\in\Gw$ the matrix $A(x):=\big[a^{ij}(x)\big]$ is symmetric and that the real quadratic form
\be\label{ellip}
 \xi \cdot A(x) \xi := \sum_{i,j =1}^N \xi_i a^{ij}(x) \xi_j \qquad
 \xi \in \Real ^N
\end{equation}
is positive definite. Moreover, throughout the paper it is assumed that $P$ is locally uniformly elliptic, and that locally, the coefficients of $P$ are sufficiently regular in $M$ such that standard elliptic (local) regularity results hold true. Our results
hold for example when  $A$ and $f$ are locally H\"{o}lder continuous, $b,\,\tilde{b}$ are Borel measurable vector fields in $M$ of class
$L^p_{\mathrm{loc}}(M)$,  and $c \in L^{p/2}_{\mathrm{loc}}(M)$ for some $p > N$.   In fact, we need to assume further local regularity on the coefficients that guarantee the existence of the limit
 \begin{equation}\label{eq_11}
 \lim_{x\to x_0}\frac{u(x)}{v(x)}\,,
\end{equation}
where $u$ and $v$ are positive solutions of the equation $Pu=0$ in a punctured neighborhood of any $x_0\in M$, and the limit might be $\infty$ (for sufficient conditions that guarantee it, see for example \cite{Fuchs} and references therein).

The formal adjoint $P^*$ of the operator $P$ is defined on its natural space $L^2(M, \dnu)$. When $P$ is in divergence form (\ref{div_P}) and $b = \tilde{b}$, the operator
\be
\nonumber
Pu = - \div \left[ \big(A \grad u + u b\big) \right] + b \cdot \grad u + c u,
\ee
is {\em symmetric} in the space $L^2(M, \dnu)$. Throughout the paper, we call this setting the {\em symmetric case}.

\medskip

By a solution $v$ of the equation $Pu=0$ in a domain $\Gw\subset M$, we mean $v\in W^{1,2}_{\loc}(\Gw)$ that satisfies  the equation $Pu=0$ in $\Gw$ in the {\em weak sense}. Subsolutions and supersolutions are defined similarly. We denote the cone of all positive solutions of the
equation $Pu = 0$ in $\Omega$ by $\mathcal{C}_{P}(\Gw)$. We say that $P$ is {\em nonnegative in} $\Gw$ (and denote it by $P\geq 0$ in $\Gw$) if $\mathcal{C}_{P}(\Gw)\neq \emptyset$.
We recall that in the symmetric case, by the Allegretto-Piepenbrink theorem,  $P\geq 0$ in $\Gw$ if and only if the associated quadratic form is nonnegative on $\core$ (see for example \cite{pinch2}).
\medskip

{\bf Throughout the paper we always assume that $P\geq 0$ in $M$.}

\medskip

\begin{defi}\label{def_Green}{\em
A function $G_P^M:M\times M \to [-\infty,\infty]$ is said to be a {\em Green function (fundamental solution)} of the operator $P$ in $M$ if for any  $x$, $y\in M$
$$P(x,\partial_x)G_P^M(x,y)=\gd_y(x) \quad \mbox{and} \quad  P^\star(y,\partial_y)G_P^M(x,y)=\gd_x(y)  \qquad \mbox{in } M,$$
and
\begin{equation}\label{symm}
    G_{P^\star}^{M}(x,y)=G_{P}^{M}(y,x) \qquad \forall x,y\in M,
\end{equation}
where $\gd_{z}$ denotes the Dirac distribution at $z \in M$.

A positive Green function $G_P^M(x,y)$ is said to be a {\em positive minimal Green function} of $P$ in $M$ if any other positive Green function $\hat{G}_P^M(x,y)$ of $P$ in $M$  satisfies
$0< G_P^M(x,y)\leq \hat{G}_P^M(x,y)$ in $M\times M$.
 }
\end{defi}
The aim of the present article is to study the existence, uniqueness, and behaviour of a certain type of Green functions for a general (not necessarily symmetric) nonnegative elliptic operator of the form \eqref{div_P} in $M$. We recall that the existence of a fundamental solution for differential operators with {\em constant coefficients} has been proved by B.~Malgrange and L.~Ehrenpreis in \cite{Eh,ML} and for elliptic operators with {\em analytic} coefficients by F.~John \cite{FJ} using the {\em unique continuation property}.

\medskip

We are motivated by the celebrated paper of Peter Li and Luen-Fai Tam \cite{LT} who constructed a {\em symmetric} Green function for the Laplace-Beltrami operator $P:=-\Gd$ on a {\em complete}, noncompact, smooth  Riemannian manifold $M$.

In order to explain the Li--Tam result, we recall the standard construction of the positive minimal Green function for a nonnegative operator $P$ in $M$.

Let $\{M_{j}\}_{j=1}^{\infty}$ be a {\em (compact) exhaustion} of  $M$,
i.e. a sequence of smooth, relatively compact domains in $M$ such that
$M_1\neq \emptyset$, $M_j\Subset M_{j+1}$ and
$\cup_{j=1}^{\infty}M_{j}=M$. For every $j\geq 1$, let $G_P^{M_j}(x,y)$ be the Dirichlet Green function of $P$ in
 $M_j$. By the generalized maximum principle, $\{G_P^{M_j}(x,y)\}_{j=1}^{\infty}$ is
an increasing sequence of positive functions which either converges in $M\times M$ to $G_P^{M}(x,y)$, the
{\em positive minimal Green function} of the operator $P$ in $M$ (this is the {\em subcritical} case, see Definition~\ref{critical}), or
 \begin{equation}\label{Ginfty}
    \lim_{j\to\infty}G_P^{M_j}(x,y)=\infty.
 \end{equation}
If \eqref{Ginfty} holds, then we say that $P$ is {\em critical} in $M$ (for more details see Section~\ref{sec_1.5}).

Li and Tam \cite{LT} modified the above construction for the special case of a {\em critical} Laplace-Beltrami operator $P := -\Gd$ on a complete smooth manifold, by subtracting from the above sequence an appropriate sequence $\{a_j\}_{j=1}^{\infty}$ of positive numbers.  It turns out that the sequence  $\{G_{-\Gd}^{M_j}(x,y)-a_j\}_{j=1}^{\infty}$ admits a subsequence that converges to a symmetric Green function, which we call a {\em Li--Tam Green function}. Moreover, Li and Tam proved that such a Green function satisfies certain boundedness properties.

\medskip

In the present paper, we use a modification of the Li--Tam's construction to obtain, for a general critical operator $P$ of the form \eqref{div_P}, the existence of a Green function that satisfies certain boundedness properties (see Theorem~\ref{maintheorem}).  Moreover, we prove in  Theorem~\ref{FAE} that the obtained Green function is unique (up to an addition of a product of the unique ground states of $P$ and $P^\star$) in a certain class of Green functions, which we call
the {\em Li--Tam class of Green functions} (see Definition~\ref{LT_Class}). Moreover, we establish for the first time,
a {\em unique way} to obtain for a critical operator, a Li--Tam Green function, avoiding the need of the extraction of a subsequence as is done in \cite{LT}.

We note that the proof of Li and Tam has to be modified significantly, since in the general case, the constant function is not a solution, the weak maximum principle and the unique continuation property do not hold, and $P$ is not necessarily symmetric.

\medskip

Furthermore, we study the  behaviour  at infinity of Li--Tams's Green function (in the critical case).  This question was raised by Li--Tam's in  \cite{LT} and partially answered there under some assumptions on the manifold. In particular, the exact asymptotic behaviour of the constructed Green function is proved in \cite{LT} under the assumptions that  $M$ has nonnegative sectional curvature outside a compact set, and has only  small ends.  We extend their result for a general critical operator $P$ of the form \eqref{div_P}, under the assumption that  the corresponding Martin boundary is a singleton, but without any curvature assumption.

\medskip

The outline of  our paper is as follows. In Section~\ref{sec_1.5} we recall some definitions and basic known results, and state the main results of the paper (Theorem~\ref{maintheorem} and Theorem~\ref{FAE}). Section~\ref{sec_2} and Section~\ref{sec_FAE} are devoted to the proof of Theorem~\ref{maintheorem} and Theorem~\ref{FAE} respectively.
Section~\ref{limiting} is devoted to the study of behaviour of Green function at infinity.
Finally, the paper ends in Section~\ref{sec_3}, where we present some explicit examples of Green functions,  discuss some applications, and pose some questions related to our study.

\medskip

We conclude this section with some notation. Throughout the paper, we write $\Omega_1 \Subset \Omega_2$ if $\Omega_2$ is open, $\overline{\Omega_1}$ is
compact and $\overline{\Omega_1} \subset \Omega_2$. Let $f,g \in C(\Gw)$ be positive functions, we denote $f\asymp g$ in $\Gw$, if there exists a positive constant $C$ such that
$$C^{-1}g(x)\leq f(x) \leq Cg(x) \qquad \mbox{ for all } x\in \Gw,$$
and $f\sim g$ as $x \to p$ if
$\lim_{x\to p}\frac{g(x)}{f(x)}=1$.
The oscillation of a function $f$ in a set $K$ is denoted by $\mathrm{Osc}(f)$ in $K$. Finally, the ideal point in the one-point compactification of $M$ is denoted by $\bar{\infty}$.


\section{Statement of the main results}\label{sec_1.5}
In the present section we state the main results of the paper. To this end, we first recall the definitions of critical and subcritical operators and of a ground state (for more details on criticality theory, see \cite{MM1,YP3,pinch2} and references therein).
\begin{defi}\label{groundstate}{\em
Let $K \Subset M$. We say that $u \in \mathcal{C}_{P}(M \setminus K)$ is a {\em positive solution of the operator $P$ of minimal growth in a neighborhood of infinity in $M$}, if for
any $K \Subset K_{1} \Subset M$ with a smooth boundary and any positive supersolution $v$ of $Pw=0$ in $M \setminus K_{1}$ satisfying $v\in C((M \setminus K_{1})\cup \partial K_{1})$, the inequality
$u \leq v$ on $\partial K_{1}$ implies that $u \leq v$ in $M \setminus K_{1}$.

A positive solution $u \in \mathcal{C}_{P}(M)$ which has minimal growth in a neighborhood of infinity in
$M$ is called the \emph{(Agmon) ground state} of $P$ in $M.$
}
 \end{defi}
\begin{defi}\label{critical}{\em
The operator $P$ is said to be {\em critical} in $M$ if $P$ admits a ground state in $M$. The operator $P$ is called {\em subcritical} in $M$ if  $P\geq 0$ in $M$ but $P$ is
not critical in $M$.
 }
\end{defi}
\begin{rem}\label{altenatecritical}{\em
Let $P\geq 0$ in $M$. It is well known that the operator $P$ is critical in $M$ if and only if the equation $P u = 0$ in $M$ has a unique (up to a multiplicative constant)  positive supersolution
(see \cite{MM1,YP3}). In particular, $P\geq 0$ in $M$ is critical in $M$ if and only if $P$ does not admit a positive Green function in $M$. Further, in the critical case $\dim  \mathcal{C}_{P}(M) = 1$, and the unique positive solution (up to a multiplicative positive constant) is a ground state of $P$ in $M$. We note that $P$ is critical in $M$ if and only if $P^\star$ is critical in $M$.

On the other hand, $P$ is subcritical in $M$ if and only if $P$ admits a unique positive minimal Green function $G_{P}^{M}(x,y)$ in $M$. Moreover, for any fixed $y\in M$, the function $G_{P}^{M}(\cdot,y)$ is a positive solution of minimal growth in a neighborhood of infinity in
$M$. We recall that $G_{P^\star}^{M}(x,y)=G_{P}^{M}(y,x)$.
}
\end{rem}
\begin{rem}\label{subcritical_subdomain}{\em
If $P\geq 0$ in $M$, then $P$ restricted to any subdomain $\Gw \Subset M$ is subcritical, and hence, $P$ admits a unique positive minimal Green function
$G^{\Gw}_{P}$ of $P$ in $\Gw$.
 }
 \end{rem}
Since in the subcritical case there exists a unique positive minimal Green function, our goal is in fact to establish the existence of a Green function in the {\em critical} case. Our main result reads as follows.
\begin{thm}\label{maintheorem}
Let $M$ be a $C^2$-smooth, noncompact, second countable, and connected Riemannian manifold of dimension $N\geq 2$, and let $P$ be a second-order elliptic operator of the form \eqref{div_P} which is critical in $M$. Denote by $\Gf$ and $\Gf^\star$ the ground states of $P$ and $P^\star$, respectively. Then
\begin{enumerate}
\item The operator $P$ admits a Green function $G_{P}^{M}(x,y)$ in $M$ obtained by (a modification of) the Li--Tam construction. In particular, in the symmetric case, $G_{P}^{M}$ is symmetric.

\item Any Green function $G_{P}^{M}(x,y)$ obtained by the Li--Tam construction satisfies the following boundedness property:
For any $y\in M$ and any neighborhood $U_y$ of $y$ there exists $C>0$ depending on $U_y$ such that
\begin{equation}\label{eq_Green_GP}
    G_{P}^{M}(x,y)\leq C\Gf(x)\quad \mbox{and }\; G_{P^\star}^{M}(x,y)\leq C\Gf^\star(x)\qquad \forall x\in  M \setminus U_y.
\end{equation}

 \medskip
\item For any Green function $\hat{G}^{M}_{P}$ we have
 \begin{equation}\label{lim_inf}
  \liminf_{x \rightarrow  \bar{\infty}}  \frac{\hat{G}^{M}_{P}(x, y)}{\Gf(x)} = -\infty \quad  for \  each \  y \in M.
 \end{equation}
\item For any $z\in M $ there exists a Green function $\hat{G}^{M}_{P}(x, y)$ obtained by the Li--Tam construction such that in some neighborhood $U_z$ of $z$ we have
\begin{equation}\label{eq_Green_GP0}
    \hat{G}_{P}^{M}(x,z)<0 \qquad \forall x\in  M \setminus U_z.
\end{equation}
\end{enumerate}
\end{thm}
The proof of the above theorem is given in the next section.
\begin{defi}\label{LT_Class}{\em
For a critical operator $P$ in $M$, we define the {\em Li--Tam class} to be the set of all Green functions which can be obtained by the Li--Tam construction.
We denote the Li--Tam class by $\mathcal{G_{LT}}$.

The set of all Green functions  $\tilde{G}^{M}_{P}(x, y)$  of the form
\begin{equation}\label{LI-TAM}
 \tilde{G}^{M}_{P}(x, y) = G^{M}_{P} (x, y) + \big(\chi(x)\Phi^{\star}(y) + \Phi(x)\chi^\star(y)\big),
 \end{equation}
 where $G^{M}_{P}(x, y) \in \mathcal{G_{LT}},$ the functions $\Phi$, and $\Phi^{\star}$ are the ground states of $P$ and $P^{\star}$, respectively, and $\chi(x)$ and $\chi^\star(x)$ are solutions of the equation $Pu=0$ and $P^\star v=0$ in $M$, respectively, is called the {\em extended Li--Tam class} and  is denoted by $\mathcal{G_{ELT}}$.
 }
 \end{defi}
 \begin{rem}
{\em  By definition  $\mathcal{G_{LT}} \subset \mathcal{G_{ELT}}$.
 }
 \end{rem}
 In the following theorem we establish a necessary and sufficient condition for a Green function in $\mathcal{G_{ELT}}$ to be in $\mathcal{G_{LT}}$.
\begin{thm}\label{FAE}
Let $P$ and $M$ satisfy the conditions of Theorem~\ref{maintheorem}.
Suppose that $G^{M}_{P}(x, y)\in \mathcal{G_{LT}}$, and  $\hat G^{M}_{P} \in \mathcal{G_{ELT}}$.
Then the following are equivalent:
\begin{enumerate}
\item There exist $x_0,y_0\in M$ and a constant $C$ such that
 \begin{equation}\label{RA1}
 \hat G^{M}_{P}(x, y_0) \leq G^{M}_{P}(x, y_0) + C \Phi(x) \qquad   \forall x \in M,
 \end{equation}
 and
 \begin{equation}\label{RA2}
 \hat G^{M}_{P}(x_0, y) \leq G^{M}_{P}(x_0, y) + C \Phi^*(y) \qquad \forall y \in M.
 \end{equation}
\item  There exists a constant $C$ such that
 $$ \hat{G}^{M}_{P}(x, y) = G^{M}_{P}(x, y) + C\Phi(x) \Phi^{\star}(y) \qquad \forall x, y\in M.$$
\item There exist $x_0,y_0\in M$ and a constant $C$ such that
 \begin{equation}\label{RRA1}
  G^{M}_{P}(x, y_0) \leq \hat G^{M}_{P}(x, y_0) + C \Phi(x) \qquad \forall x \in M,
 \end{equation}
 and
 \begin{equation}\label{RRA2}
  G^{M}_{P}(x_0, y) \leq \hat G^{M}_{P}(x_0, y) + C \Phi^*(y) \qquad  \forall y \in M.
 \end{equation}
\item $\hat G^{M}_{P} \in \mathcal{G_{LT}}$.
\end{enumerate}
\end{thm}
 \begin{rem}
 {\em
 Theorem~\ref{FAE} implies that $\hat G^{M}_{P} \in \mathcal{G_{ELT}}$ is in $\mathcal{G_{LT}}$ if and only if $\hat G^{M}_{P}$  satisfies the ``minimality" conditions \eqref{RA1} and \eqref{RA2}.
  }
 \end{rem}

As a direct application of Theorem~\ref{FAE}, we have a sharp uniqueness result (cf. \cite[Corolarry 1]{LT}).
\begin{cor}\label{cor_1}
Let $P$ be a critical operator in $M$, and let $\tilde{G}^{M}_{P}$, and $G^{M}_{P}$ be two Green functions in $\mathcal{G_{LT}}$. If $\tilde{G}^{M}_{P}(x_0,y_0)= G^{M}_{P}(x_0,y_0)$ for some $x_0,y_0\in M$, then
$\tilde{G}^{M}_{P}= G^{M}_{P}$.
\end{cor}

\medskip

Once the existence and the uniqueness is established, the next natural and important problem concerns the behaviour of the obtained Green function near $\bar{\infty}$ (the ideal point in the one-point point compactification of $M$). However, this question seems to be subtle and delicate.  In fact, in  \cite{LT} the authors provide, a partial answer for this question for the Laplace-Beltrami operator under some strong assumptions on the manifolds.  Motivated by the results in \cite{LT}, we state the following problem.

\medskip

\begin{Prob} {\em Let $P$ be a critical operator in $M$, and let $G^{M}_{P}$ a Li--Tam Green function.  Does the following always hold true
$$\lim_{x\to\bar \infty}\frac{G^{M}_{P}(x, y)}{\Phi(x)}=-\infty\,?$$
  }
\end{Prob}
In Theorem~\ref{main_theorem} we provide an affirmative answer  to the above problem for \emph{quasi-symmetric} elliptic operators
under the assumption that the corresponding Martin boundary is a singleton, see Section~\ref{limiting} for more details and the proof.
This result significantly differ from the related result in \cite[Theorem~4]{LT}, and relies on criticality theory of second-order elliptic operator and the Martin representation theorem which allows us to drop the curvature condition and \emph{smallness} assumption on the ends of $M$. However, generally speaking,  our assumption on the Martin boundary implies that our manifold $M$ consists of one end.

\section{Proof of Theorem~\ref{maintheorem}}\label{sec_2}
The present section is devoted to the proof of Theorem~\ref{maintheorem}. The proof hinges on Lemma~\ref{keylemma} below.

We fix $p\in M$ and  an exhaustion $\{M_j\}_{j=1}^{\infty}$ such that $B(p, 1) \subset M_{1}$.  By Remark~\ref{subcritical_subdomain}, for any $j\geq 1$, the operator $P$ admits a unique minimal positive Green function $G_{P}^{M_{j}}$ in $M_{j}$.
\begin{lem}\label{keylemma}
Let $P$ be an elliptic operator of the form \eqref{div_P} satisfying $P (1) = 0$ in $M$, and let $K$ be a compact set in $M \setminus \{p \}$. Then the sequence of Green functions $\{G^{M_{j}}_{P}(\cdot, p)\}_{j=j_0}^{\infty}$ has uniformly bounded oscillation in $K$, where $j_0$ depends on $K$.
\end{lem}
\begin{remark}{\em
The proof of Lemma~\ref{keylemma} is along the lines of the proof of \cite[Lemma~1]{LT}, but in contrast to \cite[Lemma~1]{LT}, our proof does not rely on the unique continuation property, and therefore, the proof applies to operators of the form \eqref{div_P} defined on $M$.
}
\end{remark}
\begin{proof}[Proof of Lemma~\ref{keylemma}]
For $k\geq 1$, consider `annuli' of the form $A_{p}(k):=  M_{k} \setminus B(p, \frac{1}{k})$. For a fixed $k $ and $j>k$,
denote by $\gw_{j}(k)$ the oscillation of $\gmj$ on $A_{p}(k)$ defined by
\begin{equation}\label{oscillation}
\gw_{j}(k):=\!\!\!\! \sup_{x \in A_{p}(k)} \{ \gmj \}  -\!\!\! \inf_{x \in A_{p}(k)} \{ \gmj \}.
\end{equation}
Clearly $\oj(k) \geq 0$. Since for any compact  $K\subset M \setminus \{p \}$ there exists an annulus $A_{p}(k)$ such that $K\subset A_{p}(k)$, it suffices to prove that for any fixed $k\geq 1$, the sequence $\{\gw_{j}(k)\}_{j>k}$ is bounded.

We argue by contradiction. Suppose that there exists $k\in \mathbb{N}$ and a subsequence of $\oj(k)$ (that we do not rename), such that $\oj(k) \rightarrow \infty$.

Following \cite[Lemma~1]{LT}, we define for $j>k$ functions $h_{j}$ by
 \begin{equation}\label{modified_green}
 h_{j,k}(x) := \oj^{-1} \gmj - \oj^{-1} \inf_{z\in M_{k}} \{ G^{M_{j}}_{P}(z, p)\},
\end{equation}
where $\oj:=\gw_{j}(k)$.
 Clearly, $Ph_{j,k} = 0$ in $M_{j} \setminus \{ p \}$, and $\mathrm{Osc}(h_{j,k}) = 1$ in $A_{p}(k)$. On the other hand, by our assumptions on the coefficients of $P$,  and in light of \cite[Th\'{e}or\`{e}me~9.6]{Stampacchia} we have
$G^{M_{j}}_{P}(\cdot, p) \asymp   G^{M_{1}}_{P}(\cdot, p)$ in a punctured neighborhood of $p$, and since by our assumption
\eqref{eq_11} holds true, we have
\begin{equation}\label{newton}
G^{M_{j}}_{P}(x, p) \;\substack{\sim\\x \rightarrow p}\;
  G^{M_{1}}_{P}(x, p).
  \end{equation}
Therefore,
 $$
 h_{j,k}(x)  \;\;\substack{\sim\\x \rightarrow p}\;\; \oj^{-1} G^{M_{1}}_{P}(x, p).
 $$
Recall that by our assumption, $P(1)=0$. Therefore, \eqref{newton} and the maximum principle on the domain $M_{k}$ immediately imply
 \[
 \inf \{ h_{j,k}(x) : x \in M_{k}) \} = \inf \{ h_{j,k} : x \in \partial M_{k}) \} = 0.
 \]
Further, we claim that  $h_{j,k}$ satisfies the following estimate
\begin{equation}\label{firstestimate}
 \oj^{-1} G_{P}^{M_{k}}(x, p) \leq h_{j,k}(x) \leq \oj^{-1} G_{P}^{M_{k}}(x, p) + 1 \qquad \forall x\in M_{k},
\end{equation}
where $G_{P}^{M_{k}}(x, p)$ is the positive minimal Green function of $P$ on $M_{k}$. Indeed,  the right hand side of \eqref{firstestimate} can be easily verified by considering the following function
 \[
 h_{j,k}^\ga(x) := (1 - \alpha) h_{j,k}(x) -  \oj^{-1} G_{P}^{M_{k}}(x, p) - 1,
 \]
  where $0 < \alpha < 1$. Since $\ga>0$, it follows that  $h_{j,k}^\ga$ goes to $- \infty$ as $x \rightarrow p$. On the other hand, since the oscillation of $h_{j,k}$ on $\partial M_{k}$ is less or equal to $1$, it follows that $h_{j,k}^\ga \leq 0$ on $\partial M_{k}$. Hence, by the maximum principle, for any $0 < \alpha < 1$, $ h_{j,k}^\ga \leq 0$ in
$M_{k} \setminus \{ p\}$. Consequently, by letting $\ga\to 0$, we obtain  $h_{j,k}(x) \leq \oj^{-1} G_{P}^{M_{k}}(x, p) + 1$ in $M_{k}$.

The inequality $\oj^{-1} G_{P}^{M_{k}}(x, p) \leq h_{j,k}(x)$  follows by a similar argument using $$\tilde{h}_{j,k}^\ga(x) := (1 + \alpha) h_{j,k}(x) -  \oj^{-1} G_{P}^{M_{k}}(x, p)$$ with $\ga>0$.

Now letting $j \rightarrow \infty$ for  fixed $k$  in \eqref{firstestimate}, and using our assumption that $\oj \rightarrow \infty$, we obtain that (up to a subsequence) $h_{j, k}$ converges locally uniformly in  $M_{k} \setminus \{ p \}$ to a function $h_{k}$ which satisfies $P h_{k} = 0$ in $M_k \setminus \{ p \},$ and $0\leq h_k\leq 1.$
By a removable singularity theorem (see for example \cite{GS}), it follows that $h_{k}$ can be uniquely extended to a nonnegative function $\bar h_k$ which satisfies $P \bar h_k = 0$ in $M_k$.

\medskip

We claim that $\bar h_k  = C$ in $M_k$, where $C$ is a nonnegative constant. To see this, define
\begin{equation}\label{decreasing}
  S_{j}(r) := \sup_{x\in \partial{B(p,r)} } \{ \gmj \} \qquad \forall \ 0<r<r_0.
  \end{equation}
By the maximum principle, it turns out that $S_{j}(r)$ is a monotone decreasing function of $r$ for $0<r<r_0$. Moreover, by the maximum principle in $M_j\setminus B(p,r_0)$ we have 
 $$\sup_{x\in \partial{M_k} } \{ \gmj \} \leq \sup_{x\in \partial{B(p,r_0)} } \{ \gmj \}.$$
Now, for a given $l > k$, such that $1/l<r_0$, restrict $\bar h_k$ on
$M_k \setminus B(p, \frac{1}{l}).$  Applying the maximum principle and using the fact that $S_{j}$ is a decreasing function of $r$, we deduce that the maximum of $\bar h_k$ is attained on the interior boundary  $\partial B(p, \frac{1}{l}).$ Therefore, letting $l \rightarrow \infty,$ we conclude that $\bar h_k$ achieves its maximum at $p$, and hence, by the strong maximum principle $\bar h_k$ is a constant in $M_k$.

Next, let us define a sequence of function $f_{j, k} : =  C + 1 - h_{j, k}$ in $M_{k+1} \setminus B(p, \frac{1}{l}),$ where $l >  k$ and $j$ large.
Clearly, $\mathrm{Osc}(f_{j, k}) = 1$ in $A_{p}(k)$. Moreover, in light of \eqref{firstestimate}, the monotonicity of $S_{j}(r)$ (as a function of $r$), and the Harnack inequality, we conclude that $f_{j, k}$ is positive and locally bounded in $M_{k+1} \setminus B(p, \frac{1}{l})$. Hence (up to a subsequence) $f_{j, k}$ converges locally uniformly in
$M_{k+1} \setminus B(p, \frac{1}{l})$ to a function $f_{k}$. In fact, this subsequence converges locally uniformly in $M_{k+1}$. Moreover,
$f_{k}$ is a nonzero constant in $A_{p}(k)$. But this contradicts the fact that $\mathrm{Osc}(f_k)=\mathrm{Osc}(f_{j, k}) = 1$ in $A_{p}(k)$.
\end{proof}

We turn now to the proof of our main theorem.
\begin{proof}[Proof of Theorem~\ref{maintheorem}]
First, we make a simple reduction step. Recall that $P$ is critical in $M$ if and only if $P^{\star}$ is critical in $M$.
Using  a modified ground state transform, we transform the operator $P$ into an elliptic operator $L$ of the form \eqref{div_P} defined by
\[
L(u) := \Phi^{\star} P (\Phi u),
\]
where $\Phi$ and $\Phi^{\star}$ denote the ground state of the operator $P$ and $P^{\star}$, respectively.

Clearly, $L(1) = L^{\star}(1) = 0$, and hence, $L$ and $L^{\star}$ satisfy the weak and the strong maximum principle. Moreover, since $P$ is critical in $M$ it follows that $L$ and $L^{\star}$ are critical in $M$, and $1$ is the unique (up to a multiplicative constant)  ground state of $L$ and of $L^\star$ in $M$. Furthermore,
$\gpm(x, y)$ is a Green function of $P$ in $M$ if and only if $G_{L}^{M}(x, y) :=\dfrac{\gpm(x, y)}{  \Phi(x) \Phi^{\star}(y)}$ is a Green function of $L$ in $M$.

\medskip

Therefore, it is enough to prove the theorem for a critical elliptic operator $L$ of the form \eqref{div_P} that satisfies
$$L(1) = L^{\star}(1) = 0 \quad \mbox{ with ground states } \Gf=\Gf^\star=1.$$

(1)  The proof of the existence of a Green function for the operator $L$ is divided into two steps.

\medskip

{\bf Step~1:}
In this step we claim that for $p \in M$ fixed, there exist a subsequence of $\{G^{M_{j}}_{L}(x, p)\}$ (that we do not rename) and a sequence of real number $\alpha_{j}^{(p)}$ which depends on $p$ such that (up to a subsequence), the sequence of functions defined by
\begin{equation}\label{critical_sequence}
J^{M_{j}}_L(x, p) := G^{M_{j}}_{L}(x, p) - \alpha_{j}^{(p)}
\end{equation}
converges  locally uniformly in $M \setminus \{ p \}$ to a solution $J^{M}_L(x, p)$ of the equation $Lu=\gd_p$ in $M$.

Let $k> 1$, and  denote for $j>k$
$$
I_{j}(k) := \!\!\inf_{x \in \partial M_{k}} \!\!\! G^{M_{j}}_{L}(x, p), \qquad S_{j}(k) := \!\!\sup_{x \in \partial M_{k}} \!\!\!G^{M_{j}}_{L}(x, p),
$$
and let $\alpha_{j}^{(p)} := I_{j}(1)$.

By Lemma~\ref{keylemma} in  $A_{p} (k)$, there exists a constant
 $C=C(k)>0$ such that for large $j$ there holds
 \begin{equation}\label{uniformbound}
 \sup_{x \in \partial B(p, \frac{1}{k})} G^{M_j}_{L}(x, p) \leq C + \alpha_{j}^{(p)} \quad \mbox{and} \quad \alpha_{j}^{(p)} \leq C + I_{j}(k).
 \end{equation}
We claim that \eqref{uniformbound} implies that
 \begin{equation}\label{aplication_maximum}
 -C \leq J^{M_{j}}_L(x, p) \leq C \quad \mbox{in}  \ A_{p} (k).
 \end{equation}
Indeed, by the maximum principle in $A_{p} (k)$, we obtain in light of \eqref{uniformbound} that
 \[
 G_L^{M_{j}}(x, p)  - I_{j}(1) + C \geq I_{j}(k)  - I_{j}(1)+ C \geq 0 \qquad \forall x\in A_{p} (k),
 \]
and hence, $J^{M_{j}}_L(x, p) \geq -C$  in $A_{p}(k)$.

On the other hand, since $S_j(k)$ is a decreasing function of $k$, the maximum principle and  \eqref{uniformbound} imply
 \[
 G_L^{M_{j}}(x, p)  - I_{j}(1) -C  \leq \sup_{x \in \partial B(p, \frac{1}{k})} G^{M_j}_{L}(x, p) - I_{j}(1)- C \leq 0,
 \]
and hence, $J^{M_{j}}_L(x, p) \leq C$ in $A_{p}(k)$. So,  \eqref{aplication_maximum} is proved.

\medskip

Hence, the sequence $\{ J^{M_{j}}_L(\cdot,p) \}$ is locally uniformly bounded, and by standard elliptic regularity it is also locally equicontinuous in $M \setminus \{ p \}$. Arzel\`{a}–Ascoli theorem  and again elliptic regularity imply that there exists a subsequence of $\{ J^{M_{j}}_L(\cdot,p) \}$ which converges locally uniformly in $M \setminus \{ p \}$ to a solution  $J^{M}_L(x, p)$ of the equation $Lu=0$ in $M \setminus \{ p \}$.

\medskip

{\bf Claim:} $J^{M}_L(x, p)$ has a irremovable singularity at $p$, and
\begin{equation}\label{eqsing}
J^{M}_{L}(x, p) \;\substack{\sim\\x \rightarrow p}\;
 G^{M_{1}}_{L}(x, p).
 \end{equation}
Moreover, $LJ^{M}_L(x, p)=\gd_p(x)$.

Indeed, fix $0<\ga<1$, and $k>0$. Let $ G^{M_{2k}}_{L}$ be the Dirichlet Green function in  $M_{2k}$, and let $\bar \omega$ be the upper bound for the oscillation of $G^{M_{j}}_{L}$ on $M_{2k}\setminus M_{1}$, where $j>2k$. For such $j$ consider the function
  \begin{align*}
 f_{\ga,j}(x):= G^{M_{j}}_{L}(x,p)-\alpha_j^{(p)} +\bar \gw -(1-\ga)G^{M_{2k}}_{L}(x, p).
 \end{align*}
Clearly,  $\lim_{x\to p} f_{\ga,j}(x)=\infty$. On the other hand, $f_{\ga,j}(x)\geq 0$ on $\partial M_{2k}$. The maximum principle yields that $f_{\ga,j} > 0$ in $M_{2k}$.
Passing to the limit, first with  $\ga\to 0$ and then with $j\to \infty$, we obtain that
\begin{equation}\label{lowerbound}
    G^{M_{2k}}_{L}(x, p) \leq J^{M}_{L}(x,p)+\bar{\gw} \qquad x\in M_{2k}.
\end{equation}
Similarly, we obtain
\begin{equation}\label{upperbound}
J^{M}_{L}(x,p)\leq G^{M_{2k}}_{L}(x, p) +C \qquad x\in M_{2k},
\end{equation}
where $C$ is some positive constant.
Hence, $J^{M}_{L}(\cdot,p)$ is a positive solution in $M_1\setminus \{p\}$ which has a nonremovable singularity near $p$, and satisfies \eqref{eqsing}. Therefore, by the Riesz representation theorem,
we have $LJ^{M}_L(x, p)=\gd_p(x)$, and the claim is proved.

\medskip

{\bf Step~2:} In this step we establish the existence of a Green function $J^{M}_{L}(x, y)$ of $L $ in $M\times M$.

Let the reference point $p \in M$, and the converging sequence $J_{L}^{M_{j}}(x, p)$, be as in Step~1.  For a fixed $y \neq p$, consider, as in Step~1, a new sequence
 \begin{equation}\label{critical_sequence_y}
J^{M_{j}}_L(x, y) := G^{M_{j}}_{L}(x, y) - \alpha_{j}^{(y)},
\end{equation}
with a sequence $\{\ga_{j}^{(y)}\}$ of appropriate real numbers such that a subsequence of $J^{M_{j}}_{L}(x, p)$ and of $J^{M_{j}}_{L}(x, y)$ (which we do not rename)
converge to a solution in $M \setminus \{ p\}$ and $M \setminus \{ y \}$, respectively.

Recall that $G^{M_{j}}_{L^{\star}}(x, y) = G^{M_{j}}_{L}( y, x)$, where $L^{\star}$ denotes the formal adjoint of $L$. Using the fact that
  $L^{\star}(1) = 0$ and Lemma~\ref{keylemma}, we deduce as above that for a fixed $x\in M$ there exists a sequence of real numbers $\bar{\ga}_{j}^{(x)}$ such that
   \[
 G^{M_{j}}_{L^{\star}}(y, x) - \bar\ga_{j}^{(x)}
 \]
converges (up to subsequence) as a function of $y$ to a solution to the equation $L^{\star}u=0$  in $M \setminus \{ x \}$. Therefore,
  \begin{align*}
  J^{M}_{L}(x, p) &= \lim_{j \rightarrow \infty} J^{M_{j}}_{L}(x, p) \\
  & = \lim_{j \rightarrow \infty} \{ G^{M_{j}}_{L} (x, p) - \alpha_{j}^{(p)} \} \\
  & = \lim_{j \rightarrow \infty} \{ G^{M_{j}}_{L^{\star}} ( p, x) - \alpha_{j}^{(p)} \} \\
  & = \lim_{j \rightarrow \infty} \{ G^{M_{j}}_{L^{\star}} ( p, x) - \bar\ga_{j}^{(x)} \} + \lim_{j \rightarrow \infty} \{ \bar\ga_{j}^{(x)} - \alpha_{j}^{(p)} \}.
  \end{align*}
  Hence, the sequence   $\{ \bar\ga_{j}^{(x)} - \alpha_{j}^{(p)} \}$ converges (up to a subsequence) to a constant $C$. Also
  \[
   G^{M_{j}}_{L}(x, y) - \alpha_{j}^{(p)}=G^{M_{j}}_{L^{\star}}(y, x) - \alpha_{j}^{(p)} = \{ G^{M_{j}}_{L^{\star}}(y,x) - \bar\ga_{j}^{(x)} \} + \{ \bar\ga_{j}^{(x)} - \alpha_{j}^{(p)} \},
  \]
converges in $M\setminus \{x\}$ (up to a subsequence), and again as above $ G^{M_{j}}_{L}(x, y) - \alpha_{j}^{(p)} $ converges as a function of $x$ (up to a subsequence) for all $x\neq y$ to a function $J^{M}_{L}(x, y)$.

  The proof will be completed if we can show that if there is  another subsequence $\ga^{(p)}_{j_l}$ of $\ga^{(p)}_j$ such that
    \[
  G^{M_{j_l}}_{L}(x, y) - \alpha^{(p)}_{j_l}
  \]
  converges in $M\setminus \{y\}$, then it must converge to $J^{M}_{L}(x, y)$. Let us assume that
  \begin{equation}\label{different_limit}
  \lim_{l \rightarrow \infty} \{ G^{M_{j_l}}_{L}(x, y) - \alpha^{(p)}_{j_l} \} = K^{M}_{L}(x, y).
  \end{equation}
  Our aim is to prove that $J^{M}_{L}(x, y) = K^{M}_{L}(x, y)$. To this end, let us first assume that $J^{M}_{L}(\cdot, y)- K^{M}_{L}(\cdot, y)$ is a bounded function on $M\setminus \{y\}$, and hence a removable singularity theorem and the
  criticality of the operator  $L$  (with $1$ as the unique ground state) readily imply that
  $$ J^{M}_{L}(x, y) - K^{M}_{L}(x, y) = \mathrm{c}_y\qquad \forall x\in M,$$ where $\mathrm{c}_y$ is a constant depending on y. Furthermore, we have
  \begin{align*}
K^{M}_{L}(x,y) + \mathrm{c}_y&= J^{M}_{L}(x,y)  = \lim_{j \rightarrow \infty} \{ G^{M_{j}}_{L}(x,y) - \ga_{j}^{(p)} \} \\
     &= \lim_{j \rightarrow \infty} \{(G^{M_{j}}_{L}(x,y) - G^{M_{j}}_{L}(x,p)) +(G^{M_{j}}_{L}(x,p)-\ga^{(p)}_{j}) \}\\
&=     \lim_{l \rightarrow \infty}\{(G^{M_{j_l}}_{L}(x,y) - G^{M_{j_l}}_{L}(x,p)) +(G^{M_{j_l}}_{L}(x,p)-\ga^{(p)}_{j_l}) \}\\
&= \lim_{l \rightarrow \infty} \{G^{M_{j_l}}_{L}(x,y) - \ga^{(p)}_{j_l} \}
   = K^{M}_{L}(x,y).
  \end{align*}
  Hence, $J^{M}_{L}(x, y)= K^{M}_{L}(x, y)$ for all $x \in M$.

 Next we show that $J^{M}_{L}(\cdot, y)- K^{M}_{L}(\cdot, y)$ is indeed a bounded function on $M$. In fact, the proof of this statement follows as in  \cite[Theorem~1]{LT}. For the sake of completeness, we provide the proof.

Consider the difference between the two functions
 $ J^{M}_{L} (x, p) - K^{M}_{L}(x, y)$ as a function of $x \in M \setminus M_{k}$ for some fixed $k$ with $1 \leq 2k < j_{l}$ and for a fixed $y \in M_{k}.$

 \begin{align}\label{limit_compute}
 J^{M}_{L} (x, p) - K^{M}_{L}(x, y) & = \lim_{l \rightarrow \infty} \{ G^{M_{j_l}}_{L}(x, p) - \alpha^{(p)}_{j_l} \} - \lim_{l \rightarrow \infty} \{ G^{M_{j_l}}_{L}(x, y) - \alpha^{(p)}_{j_l} \}  \notag \\
 & = \lim_{l \rightarrow\infty} \{ G^{M_{j_l}}_{L}(x, p) - G^{M_{j_l}}_{L}(x, y) \}.
 \end{align}
Applying the  maximum principle on $M_{j_l} \setminus M_{2k}$ for large $j_l > 2k,$  we see

\begin{equation}\label{limiting_max}
\sup_{x \in M_{j_l} \setminus M_{2k}} \{ |G^{M_{j_l}}_{L}(x, p) - G^{M_{j_l}}_{L}(x, y)| \} \leq \sup_{x \in \partial M_{2k}} \{ |G^{M_{j_l}}_{L}(x, p) - G^{M_{j_l}}_{L}(x, y)| \}.
\end{equation}
Therefore, \eqref{limit_compute} and \eqref{limiting_max}, Lemma~\ref{keylemma} (for $L^\star$ and then for $L$) imply  for $y \in M_{k}$,
 \begin{equation}\label{local_harnack}
 \sup_{x \in M \setminus M_{2k}} \{ |J^{M}_{L} (x, p) - K^{M}_{L}(x, y) | \}     \leq C,
 \end{equation}
 where $C$ is a constant depending on $k$ and $p$. A similar argument shows that $J^{M}_{L}(x, y) - J^{M}_{L}(x, p)$ is bounded on $M \setminus M_{2k}$
 if $ y \in M_{k}.$  Consequently, for all $x \in M \setminus M_{2k}$ and each fixed $y \in M_{k},$ we have
$$
 |J^{M}_{L}(x, y) - K^{M}_{L}(x, y)|  \leq    | J^{M}_{L}(x, y)-J^{M}_{L}(x, p)| + |J^{M}_{L}(x, p) - K^{M}_{L}(x, y)|  \leq C_1,
$$
 where $C_1$ is a constant depends on $k$ and $p$.

On the other hand, it follows from the proof of the Claim that
$J^{M}_{L}(\cdot, y) - K^{M}_{L}(\cdot, y)$ is a bounded  function on $M_{2k}.$ This together with the above imply that $J^{M}_{L}(\cdot, y) - K^{M}_{L}(\cdot, y)$ is bounded in $M$.

  \medskip
Since $y\in M$ is an arbitrary point, this completes the proof that $J^{M}_{L}(x, y)$ is a well defined Green function of $L $ satisfying

\[
L J^{M}_{L}(x, y) = \delta_{y}(x) \quad \mbox{in} \ M.
\]
Moreover, from the construction it follows that $J^{M}_{L^{\star}}(x, y) = J^{M}_{L}(y, x)$.

\medskip

(2)  We need to show that for a fixed $y\in M$, $J^{M}_{L}$ is bounded above away from the pole $y$ (that is,  $J^{M}_{L}$ satisfies \eqref{eq_Green_GP}).

Consider the sequence
  \[
   G^{M_{j}}_{L}(x, y) - S_{j}(1)   =  (G^{M_{j}}_{L}(x, y) - I_{j}(1)) - (  S_{j}(1) - I_{j}(1) ).
  \]
Since the sequence $\{ S_{j}(1) - I_{j}(1) \} $ is bounded by Lemma~\ref{keylemma}, we deduce that $ \{ G^{M_{j}}_{L}(x, y) - S_{j}(1) \}$   converges (up to a subsequence) uniformly on compact subsets of $M \setminus \{ y \}$ to a Green function, denoted by $G^{M}_{L}$, that satisfies
  \[
 G^{M}_{L}(x, y) := J^{M}_{L}(x, y)  + C.
  \]

Define for $ k \geq 1,$
 \[
 \tilde S_{j}(k) := \sup  \{G^{M_{j}}_{L}(x, p) : x \in \partial M_{k} \}   - S_{j}(1).
 \]
By the maximum principle, the sequence $\{\tilde S_{j}(k)\}$ is decreasing as a function of $k$,  hence, $\tilde S_{j}(k) \leq \tilde S_{j}(1) = 0$ for all $k \geq 1$.
 Therefore, the maximum principle implies that  $$ G^{M_{j}}_{L}(x, p) - S_{j}(1) \leq 0\quad  \forall x \in M_j \setminus M_1.$$
 By passing to the limit, we conclude that
 $G^{M}_{L}(x, p) \leq 0$ for all $x \in M \setminus M_{1}$. Fix $y \in M_{k-1}$. In light of Lemma~\ref{keylemma}, we have for all $x \in M_j\setminus M_{k}$, and $j > k,$
 \begin{multline*}
 G^{M_{j}}_{L}(x, y) - S_{j}(1)  = \big( G^{M_j}_{L}(x, y) - G^{M_j}_{L}(x, p)\big)+ \big(G^{M_j}_{L}(x, p) - S_{j}(1)\big) \\
  \leq |G^{M_j}_{L}(x, y) - G^{M_j}_{L}(x, p)| \leq \sup_{x \in M_{j} \setminus M_k} |G^{M_j}_{L}(x, y) - G^{M_j}_{L}(x, p)| \\
  \leq  \sup_{x \in \partial M_k} |G^{M_j}_{L}(x, y) - G^{M_j}_{L}(x, p)| \leq C.
 \end{multline*}
 Therefore, we have
 $$
 G^{M}_{L}(x, y) \leq C \qquad \forall x \in M\setminus M_{k}, \; \forall y \in M_{k-1}.
 $$
  and this implies \eqref{eq_Green_GP}.

  \medskip

  (3) Suppose to the contrary that there exists $y\in M$ and $C\in \R$ such that
  \begin{equation}\label{eq_Green_GP_bb}
    \hat{G}_{L}^{M}(x,y)> C \qquad \forall x\in  M \setminus U_y,
\end{equation}
where $U_y$ is a bounded neighborhood of $y$. Then, the maximum principle implies that  $\hat{G}_{L}^{M}(\cdot,y)-C>0$ in $M$. So,  $\hat{G}_{L}^{M}(\cdot,y)-C>0$ is a positive supersolution of the equation $Lu=0$ in $M$ which is not a solution. But this contradicts the criticality of $L$ in $M$.

  \medskip

 (4) Let $\tilde{G}_{L}^{M}(x,y)$ be a Green function that satisfies \eqref{eq_Green_GP}, and fix $z\in M$. Let $C_z$ a positive constant and $U_z$ a neighborhood of $z$ such that
$$    \tilde{G}_{L}^{M}(x,z)< C_z \qquad \forall x\in  M \setminus U_z,$$
and define
$\hat{G}_{L}^{M}(x,y):= \tilde{G}_{L}^{M}(x,y)-C_z.$ Then $\hat{G}_{L}^{M}(x,y)$ is a Green function that satisfies \eqref{eq_Green_GP} and satisfies
$$    \hat{G}_{L}^{M}(x,z)<0 \qquad \forall x\in  M \setminus U_z.$$
\end{proof}

\begin{rem}\label{GP_green_function}
{\em  It follows from the proof of Theorem~\ref{maintheorem} that for any $p\in M$ we may construct a Green function $G^{M}_{P}$ such that $G^{M}_{P}(x,p)\leq 0$ for all $x\in M\setminus M_1$.
 }
 \end{rem}
 \medskip
\section{Proof of Theorem~\ref{FAE}}\label{sec_FAE}
This short section is devoted to the proof of Theorem~\ref{FAE}. As in the proof of Theorem~\ref{maintheorem},  it is enough to prove the theorem for the operator $L$, where $L:= \Phi^{\star} P \Phi$.

We first prove a simple lemma.
\begin{Lem}\label{lem5}
Let $\hat G^{M}_{L}\in \mathcal{G_{ELT}}$ and $G^{M}_L\in \mathcal{G_{LT}}$. Then  there exist $\chi$, and $\chi^{\star}$
such that $L\chi=0$, and $L^\star\chi^\star=0$ in $M$, and
\begin{equation}\label{NAS}
\hat G^{M}_{L}(x, y) = G^{M}_{L}(x, y)  + \chi(x) + \chi^{\star}(y)\qquad \mbox{in } M\times M.
\end{equation}
\end{Lem}
\begin{proof}
By definition, $\hat G^{M}_{L} \in \mathcal{G_{ELT}}$ means that there exist $\tilde G^{M}_{L}\in \mathcal{G_{LT}}$, and $\chi_{1}$, $\chi_{1}^{\star}$ such that
$$ \hat G^{M}_{L}(x,y) = \tilde G^{M}_{L}(x, y) + \chi_{1}(x) +\chi_{1}^{\star}(y),$$
where $\tilde G^{M}_{L} \in \mathcal{G_{LT}}$, $L\chi_1=0$, and $L^\star\chi_1^\star=0$ in $M$.

Therefore, in order to prove \eqref{NAS}, it is enough to show that there exist $\tilde \chi$ and $\tilde \chi^{\star}$ with
$ L \tilde \chi=0$ and $ L^\star \tilde \chi^\star=0$ in $M$ such that
\begin{equation}\label{112}
    \tilde G_{L}^{M}(x, y) - G^{M}_{L} (x, y) = \tilde \chi (x) + \tilde \chi^{\star}(y) \qquad \mbox{in } M\times M.
\end{equation}
To see this, let us consider the difference $\hat f(x, y):=\tilde G_{L}^{M}(x, y) - G^{M}_{L} (x, y)$. Clearly, $\hat f(\cdot, y)$ and $\hat f(x, \cdot)$  are solutions to $L u = 0$ and $L^{\star} v = 0$, respectively.
Following Li--Tam \cite[Theorem~2]{LT},  we define a new Green function
 $$G^{1}_{L}(x, y) := G^{M}_{L}(x, y) + \hat f(x, p) + \hat f(p, y)- \hat f(p, p),$$
where $p \in M$ is fixed.

We assert that $G^{1}_{L}(x, y) = \tilde G^{M}_{L}(x, y)$, which clearly implies \eqref{112}, and hence \eqref{NAS}.

Indeed,  let  $\tilde f (x, y) := \tilde G^{M}_{L}(x, y) - G^{1}_{L}(x, y)$. Then
$$\tilde f (x, y) = \hat{f}(x, y) - \hat{f}(x, p) - \hat{f}(p, y) + \hat{f}(p,p),$$ and $\tilde f (x, p) = \tilde f(p, y) = 0$ for any $x, y \in M$.  Since  $\tilde G^{M}_{L}, G^{M}_{L} \in \mathcal{G_{LT}}$, it follows (as in the proof of Theorem~\ref{maintheorem}) that   for any $k \in \mathbb{N}$,  $q \in M_{k}$ and
$y \in M \setminus M_{2k}$
\begin{multline*}
|\tilde f(p, y) - \tilde f(q, y)|  \leq |\tilde G^{M}_{L}(p, y) - \tilde G^{M}_{L}(q, y)| + | G^{M}_{L}(p, y) - G^{1M}_{L}(q, y)|\\  -
\hat f(p, p) + \hat f(q, p)\leq C,
\end{multline*}
where $C$ depends on $k, p,q$. On the other hand,  as a function $y$, the difference $|\tilde f(p, y) - \tilde f(q, y)|$ is bounded in $M_{2k}$.  Therefore, the criticality of $L^{\star}$ and the arbitrariness of $k$  implies that
$$ \tilde f(p, y) - \tilde f(q, y) = C(p,q) \quad \quad \forall y ,q\in M,$$
where $C(p, q)$ is a constant depending only on $p, q$. By substituting  $y = p$ in the above equation, we see that $C(p,q)=0$ and consequently,  $\tilde f(q, y) = 0$ for all $y , q \in M$. This proves the assertion and therefore also \eqref{NAS}.
\end{proof}
\begin{proof}[Proof of Theorem~\ref{FAE}]
Obviously, $(2)$ implies $(1)$ and $(3)$.
Next we show that $(1)\,\Rightarrow \,(2)$. Set $$f(x,y):=\hat{G}^{M}_{L}(x, y)-G^{M}_{L}(x, y).$$
By Lemma~\ref{lem5},
$$f(x,y)= \chi(x) + \chi^{\star}(y)\qquad \mbox{in } M\times M,$$
where $\chi$, and $\chi^{\star}$ satisfy $L\chi=0$, and $L^\star\chi^\star=0$ in $M$, respectively.
We need to prove that $f=\mathrm{constant}$.

\medskip

Next we claim for any $q\in M$ there exists $k \in \mathbb{N}$ such that for $y \in M_{k},$ the function
$$\hat G^{M}_{L}(x, q) - \hat G^{M}_{L}(x, y) $$ is a bounded function of $x$ on $M \setminus M_{2k}.$ Indeed we write for a fixed $p \in M$
\begin{multline}\label{G-bounded}
|\hat G^{M}_{L}(x, q) - \hat G^{M}_{L}(x, y) |   \leq |G^{M}_{L}(x, q) -  G^{M}_{L}(x, p)| + |G^{M}_{L}(x, p) -  G^{M}_{L}(x, y)|  \\[2mm]
 + |\chi^\star(q) - \chi^\star(y)|,
\end{multline}
and the boundedness of \eqref{G-bounded} follows similarly to the proof of \eqref{local_harnack} in Step~2 of Theorem~\ref{maintheorem}.
 Therefore, using the boundedness of \eqref{G-bounded} and \eqref{RA1} we have for a fixed $y \in M_k$
 and for all $x \in M \setminus M_{2k}$
\begin{multline*}
f(x,y)= \hat{G}^{M}_{L}(x, y)-G^{M}_{L}(x, y)  \leq |\hat{G}^{M}_{L}(x, y)- \hat G^{M}_{L}(x, y_0)|  \\[2mm]
 + |{G}^{M}_{L}(x, y_0)-G^{M}_{L}(x, y)|+ \hat{G}^{M}_{L}(x, y_0)-G^{M}_{L}(x, y_0)   \leq C,
\end{multline*}
for some $y_{0} \in M.$
This implies, for a fixed $y \in M_k,$ the function $f(\cdot, y)$ is a bounded above solution of the equation $Lu=0$ in $M$. On the other hand, by the criticality of $L$ in $M$ it follows that
any nonconstant solution $v$ of the equation $Lu=0$ in $M$ satisfies
$$\liminf_{x\to \bar{\infty}} v(x)=-\infty, \quad \mbox{ and }  \quad  \limsup_{x\to \bar{\infty}} v(x)=\infty.$$ Therefore,
$$\liminf_{x\to \bar{\infty}} f(x, y)=-\infty, \quad \mbox{ and }  \quad  \limsup_{x\to \bar{\infty}} f(x,y)=\infty,$$
but this contradicts the fact that $f(\cdot, y)$ is bounded above. Hence, $f(x,y):=F(y)$. A similar consideration concerning $f(x,y)$ as a function of $y$ keeping $x$ fixed implies $f(x,y):=G(x)$. So, $f(x,y):=G(x)=F(y)=\mathrm{constant}.$

\medskip

$(3)\; \Rightarrow \; (2)$ follows immediately by considering $\tilde f(x, y):= G^{M}_{L}(x, y) - \hat G^{M}_{L}(x, y)$ and using similar arguments as above.

$(2)\; \Leftrightarrow \; (4)$ is obvious.
This completes the proof of the theorem.
\end{proof}
\begin{remark}\label{rem12}{\em
Corollary~\ref{cor_1} concerns the uniqueness of Li--Tam's Green functions. It implies that a Li--Tam Green function is uniquely defined by the limiting process of the proof of Theorem~\ref{maintheorem}, once the limiting value at a reference point $(q,p)$ is fixed.

Indeed, without loss of generality we may consider the operator $L$.  Fix a subsequence of
$$\left\{J^{M_{j}}_L(x,y):=G^{M_j}_{L}(x,y) - \alpha_{j}^{(p)}\right\}_{j=1}^\infty$$
that converges to a Li--Tam Green function $G$ satisfying $G(q,p)=c$ for some $c\in \R$.
It was shown in the proof of Theorem~\ref{maintheorem} that any subsequence of $\{J^{M_{j}}_L\}$ admits a subsequence that converges to a Li--Tam Green function. Corollary~\ref{cor_1} implies that any limit of such a subsequence that takes the value $c$ at $(q,p)$, is equal to $G$. In particular, if a subsequence of $\{J^{M_{j}}_L\}$ converges at a point $(q,p)$, it converges everywhere to a Li--Tam Green function.
}
 \end {remark}


\section{Behaviour at infinity of the Li--Tam Green function}\label{limiting}
The present section is devoted to the study of limiting behaviour of a Green functions near $\bar{\infty}$ in the critical case. To this end we start with the following definition:
\begin{defi}
{\em
 A second-order elliptic operator $P$ is said to be \emph{quasi-symmetric} in $M$ if for some reference point $x_0 \in M,$ the Na\"im kernel
$$\theta(x, y) := \frac{G^{M}_{P}(x, y)}{G^{M}_{P}(x, x_0)G^{M}_{P}( x_0, y)}$$
satisfies
$$
\theta(x, y) \leq C \theta(y, x) \qquad \forall \ (x, y) \in (M \setminus \{ x_0 \})^2,
$$
and for some $C \geq 1.$
}
\end{defi}
With this definition in mind, we state the main theorem of the present section.
\begin{thm}\label{main_theorem}
Let $P$ a critical operator in $M$, where $P$ and $M$ satisfy the assumptions of Theorem~\ref{maintheorem}. Fix $0 \lvertneqq W \in C^{\infty}_{0}(M)$, Assume further that $P+W$ is quasi-symmetric in $M$, and that the corresponding Martin boundary is a singleton set. Then
\begin{equation}\label{main_limit}
\lim_{ x \rightarrow {\bar{\infty}}} \frac{G^{M}_{P}(x, y)}{\Phi(x)} = - \infty,
\end{equation}
where $\Phi$ is the unique ground state of $P$ in $M$ and $G^{M}_{P}$ is the Li--Tam Green function as constructed in Theorem~\ref{maintheorem}.
 \end{thm}

We recall that when $P$ is quasi-symmetric and \emph{subcritical} in $M$, then the above limit with respect to a certain positive $P$-harmonic function is well understood and in particular Ancona proved the following result.

\begin{thm}[Theorem~1, \cite{AA}]\label{ancona}
Assume that $P$ is subcritical and quasi-symmetric in $M$. Then, there exists a positive $P$-harmonic function $u$ in $M$ such that
for each fixed $y \in M,$
$$
\lim_{x \rightarrow {\bar{\infty}}} \frac{G^{M}_{P}(x, y)}{u(x)} = 0.
$$
Moreover, in the non quasi-symmetric case, the above limit might not exist for any $u\in \mathcal{C}_{P}(M)$.
\end{thm}
\begin{proof}[Proof of Theorem~\ref{main_theorem}]
The proof is divided into three small steps.

{\bf{Step 1}:} First, as in the proof of Theorem~\ref{maintheorem}, we make a simple reduction to a critical operator $L:= \Phi^{\star} P \Phi$.
Therefore, in order to prove Theorem~\ref{main_theorem}, it is enough to show that
$$
\lim_{x \rightarrow {\bar{\infty}}} G^{M}_{L}(x, y) = -\infty
$$
for each fixed $y \in M.$

\medskip

 {\bf{Step 2 :}} In this step we show that there exists a positive $L$-harmonic function $u$ in a neighborhood of $\bar{\infty}$ in $M$ such that $u(x) \rightarrow \infty$ as $x \rightarrow {\bar{\infty}}$.

\medskip

Indeed, recall that part (4) of Theorem~2.5 implies that the critical operator  $L$ in $M$  admits a Li--Tam Green function $G^{M}_{L}(x, y)$   such that
 $$
 G^{M}_{L}(x, y) < 0 \qquad \forall x \in M\setminus U_y,
 $$
where $U_y$ is a neighbourhood of $y$. Now consider a smooth compact set $K$ in $M$ such that $U_y \subset K$.
The criticality of $L$ in $M$ readily implies that  $L$ is subcritical in $M \setminus K.$ By applying Theorem~\ref{ancona} for  $L$ on $M \setminus K,$ we obtain the existence of a positive $L$-harmonic function $u$ in $M \setminus K$ such that
\begin{equation}\label{limit_green_subdomain}
\lim_{x \rightarrow {\bar{\infty}}} \frac{G_{L}^{M \setminus K}(x, z)}{u(x)} = 0,
\end{equation}
for each fixed $z \in M \setminus K,$ where $G_{L}^{M \setminus K}(x, z)$ is the unique positive minimal Green function of $L$ in $M \setminus K$. Furthermore, we claim that
 for each fixed $z \in M \setminus K$ we have
\begin{equation}\label{equiv_1}
 G^{M\setminus K}_{L}(x, z) \asymp 1
 \end{equation}
 near ${\bar{\infty}}$.

Indeed, $G^{M \setminus K}_{L}$ is a positive solution of $L u = 0$ near ${\bar{\infty}}$ and $1$ is the ground state of $L$ in $M$. In particular, $1$ is a positive solution of minimal growth near $\bar{\infty}$.
Hence, for each fixed $z \in M$ there exists a positive constant $C(z)$ such  $G^{M \setminus K}_{L}(\cdot,z) \geq C(z) $  near ${\bar{\infty}}$.

On the other hand, $G^{M \setminus K}_{L}$ is the
positive minimal Green function of $L$ in $M\setminus K$, and hence it is positive solution of minimal growth near $\bar{\infty}$. Moreover, $L (1) = 0$ in $M$. Therefore, $G^{M \setminus K}_{L} \leq C_2$ near ${\bar{\infty}}$ for some positive constant $C_2$. Hence, the claim follows.

\medskip

Therefore, \eqref{limit_green_subdomain} and \eqref{equiv_1} immediately  yields $u(x) \rightarrow \infty$ as $x \rightarrow {\bar{\infty}}.$

\medskip

{\bf{Step 3 :}} In this final step we use our hypothesis that the Martin boundary of $M$ with respect to $L + W$ is a singleton. This assumption implies that
the Martin boundary of $M \setminus K$ with respect to $L$ consist of two components, the Euclidean boundary $\partial K$ and  $\partial_{m}{\bar{\infty}}$ (singleton set). By invoking
Martin's integral representation of positive $L$-harmonic functions, it follows that any $L$-harmonic functions $v$ in $M \setminus K$ is uniquely represented  as
\begin{equation}\label{martin_represent}
v(x) =  \int_{\partial K} \mathcal{K}_0(x, z) \ \mbox{d}\sigma(z) +  \alpha \mathcal{K}_ {{\bar{\infty}}}(x),
\end{equation}
where $0\leq \alpha \leq1$, $\sigma$ is a nonnegative finite measure on $\partial K$, $\mathcal{K}_0(x, z)$ is the Martin kernel corresponding to $z\in \partial K$, and $\mathcal{K}_{{\bar{\infty}}}(x)$ is the {\em unique} Martin kernel corresponding to $\bar{\infty}$, (recall that $\bar{\infty}$ is the ideal point in the one point compact compactification of $M$).
\medskip

We claim that if $\sigma\neq 0$, then $\int_{\partial K} \mathcal{K}_0(x, z) \ \mbox{d}\sigma(z)  \asymp 1$ for all $x$ near $\bar{\infty}$ and for $z \in \partial K$.

Indeed, since $1$ is a ground state of $L$ in $M$ it follows  there exists $C$ and $j=j(M,K,\gs)$, such that $K\Subset M_j$ and
$$ \int_{\partial K} \mathcal{K}_0(x, z) \ \mbox{d}\sigma(z) \geq C \qquad \forall x \in M\setminus M_j.$$
Next,
using Martin's compactification theorem, for $x \in M\setminus K$ we have

\begin{equation}\label{martin_compact}
\frac{G^{M \setminus K}_{L}(x, z_n)}{G^{M \setminus K}_{L}(x_0, z_n)} \rightarrow \mathcal{K}_0(x, z), \qquad  \mbox{whenever} \ z_n \rightarrow z \in \partial K
\end{equation}
On the other hand, the boundary Harnack principle for $L^*$ implies that there exist $\delta > 0$  and a constant $C > 0$ such that for a fixed $z_1 \in B(z, \delta)\cap (M\setminus K)$, we have
$$
\frac{G^{M \setminus K}_{L}(x, z_n)}{G^{M \setminus K}_{L}(x_0, z_n)} \leq C \frac{G^{M \setminus K}_{L}(x, z_1)}{G^{M \setminus K}_{L}(x_0, z_1)},
$$
for all $x$ near $\bar{\infty}$, where $C$  depends on $\delta$. Since $G^{M \setminus K}_{L}(x, z_1)$ is a solution of \emph{minimal growth} near $\bar{\infty},$ we conclude

$$
u_{n}(x) : = \frac{G^{M \setminus K}_{L}(x, z_n)}{G^{M \setminus K}_{L}(x_0, z_n)} \leq C
$$
for all $x$ near $\bar{\infty}$, where $C=C(z) > 0$ is independent of $z_n$. Consequently,  \eqref{martin_compact} yields that $\mathcal{K}_0(x, z) \leq C$ for all such $x$. Using the smoothness and compactness of $K$ it follows that in fact,
$ \int_{\partial K} \mathcal{K}_0(x, z) \ \mbox{d}\sigma(z) \leq C$  near $\bar{\infty}$, where  $C$ depends only on the compact set $K$.
\medskip

Moreover, by Step 2 there exists a positive $L$-harmonic function $u(x)$ which goes to $\infty$ as $x\to \bar{\infty}$.  Now, the previous conclusion readily implies
that
\begin{equation}\label{eq_infty}
\mathcal{K}_{\bar{\infty}}(x) \rightarrow \infty \qquad \mbox{as} \ x \rightarrow \bar{\infty}.
\end{equation}

Recall that  for a fixed $y \in M$, the function  $-G^{M}_{L}(x, y)$  is a positive $L$-harmonic function in $M \setminus K$ satisfying
$$
\limsup_{x \rightarrow \bar{\infty}} (- G^{M}_{L}) = \infty.
$$
Therefore, using \eqref{martin_represent}, it follows that
\begin{equation}\label{green_martin}
-G^{M}_{L}(x, y) =  \int_{\partial K} \mathcal{K}_0(x, z) \ \mbox{d}\sigma(z) +  \alpha_0 \mathcal{K}_ {{\bar{\infty}}}(x),
\end{equation}
where $\alpha_0 >0$. Thus, \eqref{green_martin} and \eqref{eq_infty} readily imply
$$
\liminf_{x \rightarrow \bar{\infty}} (- G^{M}_{L}) = \infty.
$$
This completes the proof.
\end{proof}
\section{Examples and Concluding Remarks}\label{sec_3}
In the present section we present several examples of Green functions which satisfies \eqref{eq_Green_GP}, and discuss some questions that arise in our study.
\begin{Def}{\em
We call the set of Green functions that satisfy \eqref{eq_Green_GP} the {\em class of relatively bounded above Green functions} and denote by $\mathcal{G_{BA}}$.
}
\end{Def}
Clearly, by Theorem~\ref{maintheorem} we have $\mathcal{G_{LT}} \subset \mathcal{G_{BA}}$.
\begin{example}
{\em
Consider the critical operator $P:= -\Gd=-\frac{\mathrm{d}^2}{\mathrm{d}x^2}$ in $M = \R^1$. A straightforward computation shows that a Green function $G_{-\Gd}^{\R^1}(x,y)\in \mathcal{G_{BA}}$
is given by
$$G_{-\Gd}^{\R^1}(x,y)=-\frac{1}{2}|x-y| + C,$$
where $C\in \R$ is a constant (but note that Theorem~\ref{maintheorem} is proved only for $N \geq 2$).

Similarly, a Green function $G_{-\Gd}^{\R^2}(x,y)\in \mathcal{G_{LT}}$ is given by
$$G_{-\Gd}^{\R^2}(x,y)=-\frac{1}{2\pi}\log|x-y| + C.$$
}
\end{example}

\medskip

Next, we present another $1$-dimensional example. Although Theorem~\ref{maintheorem} is proved only for $N \geq 2$, the 1-dimensional example below gives us the idea how to construct a nontrivial behaviour of a Green function in higher dimension.
\begin{example}
{\em
Let $M = (0, \infty)$ and consider the critical Hardy operator $P:= -\frac{\mathrm{d}^2}{\mathrm{d}x^2} - \frac{1}{4x^2}$. Note that $\Phi(x)=x^{1/2}$ is the ground state of $P$ in $M$. We construct a Green function for $P$ in $M$.

Define an exhaustion $\{ M_{j} \}_{j = 1}^{\infty}$ of $M$  by $M_{j} := (\frac{1}{j}, j).$ It can be shown easily that the Dirichlet Green function $G^{M_j}(x, 1)$ of $P$ in $M_j$ is given by

\begin{equation}\label{minimal_green}
G^{M_j}_{P} (x, 1) =\frac{1}{2}\big(\log j  -  |\log x|\big)x^{1/2}.
  \end{equation}
Clearly $G^{M_j}_{P}(x, 1) \rightarrow \infty$ as $j \rightarrow \infty.$ Therefore, we need to subtract a sequence of the form $\{a_{j}(x)=a_j x^{1/2}\}$ of constants times the ground state  $\Phi(x)=x^{1/2}$  such that
\[
\{ G^{M_j}_{P}(x, 1) - a_{j}(x) \}_{j = 1}^{\infty}
\]
converges to a Green function of $P$ in $M$. Note that  $Pa_{j}(x) = 0$. Choose $a_{j}(x) = \frac{1}{2} (\log j) x^{1/2}$, then the above condition is satisfied, and we obtain
a Green function $G^{M}_{P}\in \mathcal{G_{LT}}$ given by
\begin{equation}\label{hardy_green}
G^{M}_{P} (x, 1) =-\frac{1}{2} |\log x| x^{1/2}.
\end{equation}
Clearly, $\lim_{x \to \infty}\frac{G^{M}_{P} (x, 1)}{x^{1/2}}= -\infty$, and $\lim_{x \to 0}\frac{G^{M}_{P} (x, 1)}{x^{1/2}}= -\infty.$

\medskip

We note that
$$ \lim_{x \rightarrow \infty} G^{M}_{P}(x, 1) = -\infty, \qquad \mbox{while } \lim_{x \rightarrow 0} G^{M}_{P}(x, 1) = 0.$$
}
\end{example}
Let $P$ be critical in $M$. Recall that by Theorem~\ref{maintheorem}, for any $G_{P}^{M}(x,y)\in \mathcal{G_{BA}}$, we have
$$\liminf_{x \rightarrow  \bar{\infty}}\frac{ G^{M}_{P}(x, y)}{\Phi(x)}= -\infty,$$
where $\Phi$ is the ground state of $P$ in $M$. On the other hand, in all the above examples, the Green functions are not only finally negative, but also
$$\lim_{|x|\to\bar \infty}\frac{ G^{M}_{P}(x,y)}{\Phi(x)}=-\infty.$$

Let us consider one more example.
\begin{example}\label{ex_4}
{\em
Let $M = \mathbb{R}^{N} \setminus \{ 0\}$, where $N \geq 3$, and consider the critical Hardy operator $$P:= -\Delta - \frac{(N-2)^2}{4}\frac{1}{|x|^2}\,.$$
The two linearly independent, positive, radial solutions of the equations $Pu=0$ near $0$ and near $\infty$ are $v_1(x)=|x|^{(2-N)/2} $,  $v_2(x)=|\log |x|| |x|^{(2-N)/2} $, and $v_1$ is the corresponding ground state. It follows from \cite[Lemma~8.5]{Fuchs} that any positive solution $v$ of
the equation $P u = 0$ in a punctured neighborhood of $0$ or $\infty$ satisfies
\begin{equation}\label{hardy_green1}
\lim_{x\to 0}v(x)=\infty, \quad \mbox{or }  \lim_{x\to\infty}v(x)=0, \mbox{respectively}.
\end{equation}
On the other hand, Theorem~\ref{maintheorem} implies that for any point $x_{0} \in M$, there exists a Green function $G^{M}_{P}\in \mathcal{G_{BA}}$ and a neighborhood $\mathcal{U}_{x_0}$ of $x_0$
such that $G^{M}_{P}(x, x_0) < 0$ for all $x \in M \setminus \mathcal{U}_{x_0}$. Therefore, $-G^{M}_{P}(x, x_{0})$ is a positive solution of $P u (x) = 0$ near zero and near $\infty$.  Hence, by \cite{Fuchs},
$\lim_{x\to\gz}\frac{G^{M}_{P}(x, x_0)}{|x|^{(2-N)/2}}$ exists (in the generalized sense), where $\gz=0$ or $\gz=\infty$, and by Theorem~\ref{maintheorem} the limit is equal to $-\infty$ at least at one of these points. But, we do not know whether the limit is equal to $-\infty$ at {\em both} points.

\medskip

On the other hand,  $-G^{M}_{P}(x, x_{0})$ is a positive solution of $P u (x) = 0$ near $\infty$ and near $0$.  Hence, by \eqref{hardy_green1},  $G^{M}_{P}(x, x_0) \rightarrow 0$ as $x \rightarrow \infty$, while $G^{M}_{P}(x, x_0) \rightarrow -\infty$ as $x \rightarrow 0$. This is in contrast with the behaviour of Green function of $-\Delta$ in $\mathbb{R}^{2}.$
}
\end{example}
\begin{rem}
{\em
Example~\ref{ex_4} should be compared to the result of Section~\ref{limiting}. Indeed, the Martin boundary with respect to any Fuchsian type subcritical operator in $M = \mathbb{R}^{N} \setminus \{ 0\}$ consists of two Martin points, and  therefore, Theorem~\ref{main_theorem} is not applicable.
}
\end{rem}
This leads us to formulate the following two problems.
\begin{Prob} {\em Let $G^{M}_{P}\in \mathcal{G_{BA}}$ or $G^{M}_{P}\in \mathcal{G_{LT}}$, and let $\{M_j\}$ be a compact exhaustion of $M$. Does the following assertion hold true?

For any $j\geq 1$ there exists $k_j > j$ such that
\begin{equation}\label{limsup}
  G^{M}_{P}(x, y) < 0 \qquad  \forall x\in M_j \mbox{ and } \forall y \in M\setminus M_{k_j}.
  \end{equation}
  }
\end{Prob}
\begin{rem}
{\em
It is well known that for a {\em subcritical} operator $P$ on a noncompact manifold $M$, the celebrated Martin compactification gives an integral representation for all $u\in \mathcal{C}_{P}(M)$. Such a compactification is not available for a critical operators since a critical operator does not admit a positive Green function.

Nevertheless, we may define a {\em Martin  kernel} for a critical operator $P$ with respect to a Green function $G^{M}_{P}\in \mathcal{G_{BA}}$. Let $x_0\in M_1$ be a fixed reference point. There exists a Green function $G^{M}_{P}\in\mathcal{G_{BA}}$ and a neighborhood $\mathcal{U}_{x_0}$ of $x_0$ such that
$G^{M}_{P}(x_0, y) < 0$ for all $y \in M\setminus \mathcal{U}_{x_{0}}$. Therefore, the following Martin kernel
\[
\mathcal{K}^{M}_{P}(x, y) : = \frac{G^{M}_{P}(x, y)}{G^{M}_{P}(x_0, y)} \qquad  \forall y\in M\setminus U_{x_0},\;  x\in M\setminus \{y\}.
\]
is well defined.
}
\end{rem}
If there exists a Green function $G^{M}_{P} \in\mathcal{G_{BA}}$ which in addition satisfies \eqref{limsup}, then we have
 \begin{cor}\label{tend_ground_state}
Let $M$ be a $C^2$-smooth noncompact Riemannian manifold of dimension $N$, and let $P$ be an operator of the form \eqref{div_P}  which is critical in $M$.
Suppose that there exists a Green function $G^{M}_{P} \in \mathcal{G_{BA}}$ such that \eqref{limsup} holds true. Then
 the corresponding Martin kernel $\mathcal{K}^{M}_{P}$  satisfies
\begin{equation}\label{eq_limit_ground}
\lim_{y \rightarrow  \bar\infty} \mathcal{K}^{M}_{P}(x, y) = \Phi(x),
\end{equation}
where $\Phi$ is the ground state of $P$ satisfying $\Phi(x_0)=1$.
\end{cor}
\begin{proof}
Let $G^{M}_{L} \in \mathcal{G_{BA}}$ such that \eqref{limsup} holds true. Let $x\in M_{j}$ for some $j\in\mathbb{N}$.
By our assumption, there exists $k_{j} > j$ such that $G^{M}_{P}(x, y)$ are negative for  all
$x \in M_j$ and $y\in  M\setminus  M_{k_j}$. Consider the Martin kernel
\[
\mathcal{K}^{M}_{L} (x, y) = \frac{G^{M}_{L}(x, y)}{G^{M}_{L}(x_0, y)}\,.
\]
Clearly, $\mathcal{K}^{M}_{L} (x_0, y) = 1$. Moreover, for any fixed $j$ and $y\in M\setminus M_{k_j}$, the function $\mathcal{K}^{M}_{L}(\cdot,y)$ is a positive solution of the equation $Pu=0$ in $M_j$.

Since $P$ is critical, it follows from the Harnack principle and a standard diagonalization argument that for any sequence $y_n\to \bar{\infty}$, there exists a subsequence $\{y_{n_\ell}\}$, such that the sequence
\begin{equation*}\label{eq_limit_ground1}
\left\{\mathcal{K}^{M}_{P}(x, y_{n_\ell})\right\}_{\ell=1}^\infty
\end{equation*}
converges to a positive solution of the equation $Pu=0$ in $M$, and \eqref{eq_limit_ground} follows by the uniqueness of the ground state.
\end{proof}
By Theorem~\ref{maintheorem} we have $\mathcal{G_{LT}} \subset \mathcal{G_{BA}}$, and hence $\mathcal{G_{ELT}} \cap \mathcal{G_{BA}} \neq \emptyset$. It is natural to pose the following problem.
\begin{Prob}
{\em
Characterize the class $\mathcal{G_{BA}}$.
}
\end{Prob}


\medskip

 \begin{center}{\bf Acknowledgments} \end{center}
D.~G. was supported in part at the Technion by a fellowship of the Israel Council for Higher Education. The authors acknowledge the support of the Israel Science Foundation (grants No. 963/11 and 970/15) founded by the Israel Academy of Sciences and Humanities.


\begin{thebibliography}{99}


\bibitem{AA} A.~Ancona, Some results and examples about the behaviour of harmonic functions and Green's functions with respect to second order elliptic operators,
\emph{Nagoya Math.~J.} \textbf{165} (2002), 123--158.

\bibitem{Eh} L.~Ehrenpreis, Solution of some problems of division. I. Division by a polynomial of derivation, \emph{Amer. J. Math.} \textbf{76} (1954), 883--903.

\bibitem{GS} D.~Gilbarg, and J.~Serrin, On isolated singularities of solutions of second order elliptic differential equation, \emph{J.Analyse Math.} \textbf{4} (1956), 309--340.

\bibitem{FJ} F.~John, \emph{Plane Waves and Spherical Means Applied to Partial Differential Equations}, Interscience, New York, 1955. Reprinted, Springer-Verlag, Berlin, 1981.

\bibitem{LT} P.~Li, and  L.-F.~Tam,  Symmetric Green's functions on complete manifolds,  \emph{Amer. J. Math.} \textbf{109} (1987), 1129--1154.

\bibitem{ML} B.~Malgrange, Existence et approximation des solutions des \'equations aux d\'eriv\'ees partielles et des \'equations de convolution, \emph{Ann. Inst. Fourier, Grenoble}
\textbf{6} (1955), 271--355.

\bibitem{MM1} M.~Murata,  Semismall perturbations in the Martin theory for elliptic equations,
\emph{Israel J.~Math.} \textbf{102} (1997), 29--60.

\bibitem{YP3} Y.~Pinchover,  On positive solutions of second order elliptic equations, stability results and classification,  \emph{Duke Math J.} \textbf{57} (1988), 955--980.

\bibitem{Fuchs} Y.~Pinchover, On positive Liouville theorems and asymptotic behavior of solutions of Fuchsian type elliptic operators, \emph{Ann.  Inst.  H. Poincar\'{e}. Anal. Non Lin\'{e}aire} \textbf{11} (1994), 313--341.

\bibitem{pinch2} Y.~Pinchover, Topics in the theory of positive solutions of second-order elliptic and parabolic partial
differential equations, in \emph{Spectral Theory and Mathematical Physics: A Festschrift in Honor of Barry Simon's 60th Birthday},
eds. F.~Gesztesy, et al., Proceedings of Symposia in Pure Mathematics \textbf{76}, American Mathematical Society, Providence, RI, 2007, 329--356.

\bibitem{Stampacchia} G.~Stampacchia, Le probl\`{e}me de Dirichlet pour les \'{e}quations elliptiques
du second ordre \`{a} coefficients discontinus, \emph{Annales de l’institut Fourier} \textbf{15} (1965), 189--257.

\end{thebibliography}
\end{document}